\documentclass[12pt]{amsart}

\allowdisplaybreaks

\usepackage{diagbox}
\usepackage{fullpage}

\usepackage{amsfonts,amssymb,amscd,amsmath,enumerate,verbatim,color,xcolor}
\usepackage[latin1]{inputenc}
\usepackage{latexsym}
\usepackage{bm}

\usepackage{graphicx} 

\usepackage{mathptmx}

\usepackage{hyperref}
\hypersetup{colorlinks,linkcolor=blue ,citecolor=blue, urlcolor=blue}

\usepackage{amsfonts,amssymb,amscd,amsmath,enumerate,verbatim,color}
\usepackage[latin1]{inputenc}
\usepackage{tikz}
\usetikzlibrary{positioning}
\usepackage{tabularx,array}
\newcolumntype{C}{>{\centering\arraybackslash}X}

\newtheorem{thm}{Theorem} [section]
 
\newtheorem{prop}[thm]{Proposition} 
\newtheorem{lemma}[thm]{Lemma}
\newtheorem{cor}[thm]{Corollary}

\theoremstyle{definition}
\newtheorem{example}[thm]{Example}
\newtheorem{conjecture}[thm]{Conjecture}

\newtheorem{remark}[thm]{Remark}
\numberwithin{thm}{section}  
\numberwithin{equation}{section} 

\def\Pc{{\mathcal P}}

\def\NZQ{\mathbb}               

\def\ZZ{{\NZQ Z}}
\def\RR{{\NZQ R}}

\def\eb{{\mathbf e}}

\def\opn#1#2{\def#1{\operatorname{#2}}} 
\opn\gr{gr}

\def\Ac{{\mathcal A}}

\def\Mc{{\mathcal M}}

\def\Pc{{\mathcal P}}
\def\Qc{{\mathcal Q}}

\def\Cc{{\mathcal C}}

\def\int{{\rm int}}

\title{Facet numbers of non-centrally symmetric reflexive polytopes arising from posets}
\author{AKI MORI, KENTA MORI and HIDEFUMI OHSUGI}
\date{}
\address{Aki Mori,
	Center for Physics and Mathematics, Institute for Liberal Arts and Sciences, Osaka Electro-Communication University, Neyagawa, Osaka 572-8530, Japan} 
\email{a-mori@osakac.ac.jp}

\address{Kenta Mori,
	Department of Mathematical Sciences,
	School of Science,
	Kwansei Gakuin University,
	Sanda, Hyogo 669-1330, Japan}
\email{k-mori@kwansei.ac.jp}

\address{Hidefumi Ohsugi,
	Department of Mathematical Sciences,
	School of Science,
	Kwansei Gakuin University,
	Sanda, Hyogo 669-1330, Japan} 
\email{ohsugi@kwansei.ac.jp}


\begin{document}

\begin{abstract}
Twinned chain polytopes form a broad class of non-centrally symmetric reflexive polytopes and exhibit intriguing structures.
In the present paper, we show that the number of facets of $d$-dimensional twinned chain polytopes is at most $6^{d/2}$.
In case $d$ is even, the equality holds if and only if the polytope is isomorphic to a free sum of $d/2$ copies of del Pezzo polygons.
This result contributes a partial answer to Nill's conjecture: the number of facets of a $d$-dimensional reflexive polytope is at most $6^{d/2}$.
\end{abstract}

\maketitle

\section{Introduction}

A convex polytope $\Pc \subset \RR^d$ is called a {\it lattice polytope} if any vertex of $\Pc$ belongs to $\ZZ^d$. 
A {\it reflexive polytope} is a lattice polytope that contains the origin in its interior and has the property that its dual polytope is also a lattice polytope.
The concept of reflexive polytopes was introduced by Batyrev \cite{mirror} in the study of Gorenstein toric Fano varieties.
Since then, reflexive polytopes have been extensively studied from various perspectives, including combinatorics, commutative algebra and algebraic geometry.
It is known that there are only finitely many reflexive polytopes in each dimension, up to unimodular equivalence \cite{LZ}. 
Furthermore, if the dimension $d$ satisfies $d \leq 4$, reflexive polytopes can be enumerated \cite{KS, KS2}.

Let $N(\Pc)$ be the number of facets of a lattice polytope $\Pc$.
Nill proposed the following conjecture:

\begin{conjecture}[{\cite[Conjecture 5.2]{Nillconj}}] \label{Nconj}
Let $\Pc$ be a $d$-dimensional reflexive polytope.
Then $N(\Pc) \le 6^{d/2}$.    
\end{conjecture}

In Nill's original paper, this conjecture was stated as an upper bound on the number of vertices.
For dimensions $d \leq 4$, this conjecture is supported by the database of reflexive polytopes: $N(\Pc) \le 6$ if $d=2$, $N(\Pc) \le 14 = \lfloor 6^{3/2} \rfloor$ if $d =3$, and $N(\Pc) \le 36$ if $d=4$ \cite{KS3}.
A rigorous proof is known only for $d = 2$ \cite[Corollary 3.3]{Nillconj}. 
If $\Pc$ is a pseudo-symmetric reflexive simplicial $d$-dimensional polytope, then Conjecture~\ref{Nconj} holds, and the maximum $6^{d/2}$ is attained if and only if $\Pc$ is a free sum of $d/2$ copies of del Pezzo polygons \cite{Nillps}.
Additionally, Conjecture~\ref{Nconj} has been proven for symmetric edge polytopes arising from join graphs and connected bipartite graphs \cite{MMO}.

Let $(P, \leq_P)$ be a finite partially ordered set (``poset'' for short) on $P =\{p_1,\dots,p_d\}$.
An {\it antichain} of $P$ is a subset $W \subset P$ such that
$p_i$ and $p_j$ are incomparable for any $p_i,p_j \in W$ with $i\neq j$.
In particular, the empty set $\emptyset$ as well as each singleton $\{p_i\}$ is an antichain of $P$.
Let $\Ac(P)$ denote the set of antichains of $P$.
For each $A \in \Ac(P)$, we set $\rho(A) = \sum_{p_i \in A} \eb_i \in \RR^d$,
where $\eb_1, \ldots, \eb_d$ are the canonical unit coordinate vectors of $\RR^d$.  
For example, $\rho(\emptyset)$ is the origin of $\RR^d$.
Stanley \cite{Sta} defined a $d$-dimensional lattice polytope called the {\it chain polytope} of a poset $P$ as 
$$
\Cc(P) = \textrm{conv} (\{\rho(A) : A \in \Ac(P)\}) \subset \RR^d.
$$
For two lattice polytopes $\Pc, \Qc \subset \RR^d$, we define
$$
\Gamma(\Pc, \Qc) = \textrm{conv}(\Pc \cup (-\Qc)) \subset \RR^d.
$$
Let $(Q, \leq_Q)$ be a poset on $Q = \{q_1,\ldots,q_d\}$
with $P \cap Q = \emptyset$.
Then the polytope $\Gamma(\Cc(P), \Cc(Q))$ is called the {\it twinned chain polytope} \cite{T}.
It is known \cite{OH2} that $\Gamma(\Cc(P), \Cc(Q))$ is a reflexive polytope.

The study of $\Gamma(\Pc, \Qc)$ was inspired by research on centrally symmetric configurations of integer matrices \cite{OH}.
As in the case of symmetric edge polytopes \cite{MHNOH}, many reflexive polytopes exhibit the structure of $\Gamma(\Pc, \Qc)$.
In particular, when $\Pc$ and $\Qc$ are lattice polytopes arising from posets, important properties such as normality, reflexivity, unimodular equivalence and the Ehrhart polynomials of these polytopes have been studied in \cite{HM, HMT, HMT2, T}.
In \cite{T}, it was shown that twinned chain polytopes $\Gamma(\Cc(P), \Cc(Q))$ are locally anti-blocking, and hence the $h^*$-polynomials of $\Gamma(\Cc(P), \Cc(Q))$ are $\gamma$-positive.
The $\gamma$-positivity of $h^*$-polynomials of locally anti-blocking reflexive polytopes is studied in \cite{OT2}.
On the other hand, we define
$$
\Omega(\Pc, \Qc) = \textrm{conv}(\Pc \times \{1\} \cup (-\Qc) \times \{-1\}) \subset \RR^{d+1}.
$$
The polytope $\Omega(\Pc, \Qc)$ can be regarded as a variation of the Cayley sum of $\Pc$ and $\Qc$.
In particular, $\Omega(\Pc, \Pc)$ is known as the {\em Hansen polytope}~\cite{H} 
when $\Pc$ is the {\em stable set polytope} of a graph. 
In~\cite{FHSZ}, it is shown that Kalai's $3^d$ conjecture~\cite{K} holds for Hansen polytopes of split graphs.
Combinatorial properties of $\Omega(\Cc(P), \Cc(Q))$ and $\Omega(2\Cc(P), 2\Cc(Q))$ are studied in \cite{HT} and \cite{CFS}, respectively.   
It is known \cite{HT} that $\Omega(\Cc(P), \Cc(Q))$ is a reflexive polytope.
Moreover, $\Gamma(\Cc(P), \Cc(Q))$ arises as a projection of $\Omega(2\Cc(P), 2\Cc(Q))$.
In this sense, twinned chain polytopes form a natural class of reflexive polytopes reflecting certain structural features of $\Omega(2\Cc(P), 2\Cc(Q))$ while being connected to the combinatorics of posets.

In the present paper, we investigate the number of facets of twinned chain polytopes and confirm Conjecture \ref{Nconj} for these polytopes.
We recall some notation about posets.
Given a subset $W$ of $[d]:=\{1,2,\dots,d\}$, the \textit{induced subposet} of $P$ on $W$ is the poset $(P_W, \leq_{P_W})$ on 
$P_W = \{p_i : i \in W\}$
such that $p_i \leq_{P_W} p_j$ if and only if $p_i \leq_{P} p_j$.
The {\it ordinal sum} of $P$ and $Q$ is the poset $(P \oplus Q, \leq_{P \oplus Q})$ on $P \oplus Q = P \cup Q$ such that $s \leq_{P \oplus Q} t$ if (a) $s, t \in P$ and $s \leq_{P} t$, (b) $s, t \in Q$ and $s \leq_{Q} t$, or (c) $s \in P$ and $t \in Q$.
If $P$ and $Q$ are posets on {\it disjoint} sets, then the {\it disjoint union} of $P$ and $Q$ is the poset $P + Q$ on the union $P \cup Q$ such that $s \leq t$ in $P + Q$ if either (a) $s, t \in P$ and $s \leq t$ in $P$, or (b) $s, t \in Q$ and $s \leq t$ in $Q$. 
A bijective map $\varphi : P \rightarrow Q$ is
called an {\em order isomorphism} if for all $x, y \in P$, we have $x \le_{P} y$ if and only if
$\varphi(x) \le_{Q} \varphi(y)$. 
We write $P \cong Q$ whenever $P$ and $Q$ are order isomorphic.
Let ${\bf 1}$ be the singleton poset, ${\bf C}_n$ the $n$-element chain poset (i.e., a poset in which any two elements are comparable), 
and ${\bf I}_n$ the $n$-element antichain poset.
Given a poset $(P, \le_P)$, the {\it comparability graph} $G_P$
of $P$ is the graph on the vertex set $P$
whose edge set is 
$\{\{p_i,p_j\} : p_i \mbox{ and } p_j 
\mbox{ is comparable in } P\}$.
Throughout this paper, $\Gamma(\Cc(P), \Cc(Q))$ will be denoted as $\Gamma(P, Q)$.

The main result of this paper is as follows.
This result, obtained within the framework of non-centrally symmetric polytopes, adds significant value to the study of such structures.

\begin{thm}\label{main theorem} 
Let $P =\{p_1,\dots,p_d\}$ and $Q = \{q_1,\ldots,q_d\}$ be posets with $P \cap Q = \emptyset$.
Then we have 
\begin{equation}  \label{even odd eq}
    N(\Gamma (P, Q)) \le 
    \left\{
    \begin{array}{ccl}
    6^{\frac{d}{2}} & & \mbox{if } d \mbox{ is even},\\
    14 \cdot 6^{\frac{d-3}{2}} & \left(=\frac{7\sqrt{6}}{18} \cdot 6^{\frac{d}{2}} \right) & \mbox{if } d \mbox{ is odd}.\\
    \end{array}
    \right.
\end{equation}
Here $7\sqrt{6}/18 \fallingdotseq 0.95$.
In case $d$ is even, the equality holds if and only if 
$P, Q \cong {\bf I}_2 \oplus \cdots \oplus {\bf I}_2$ and
the map $\varphi: P \rightarrow Q$ defined by $\varphi(p_i) = q_i$ induces a graph  isomorphism from $G_P$ to $G_Q$.
\end{thm}

For example, if $d=1$, then $N(\Gamma (P, Q))=2$ satisfies (\ref{even odd eq}).
The posets $P, Q \cong {\bf I}_2 \oplus {\bf I}_2 \oplus {\bf I}_2$
in Figure \ref{d=6} are an example of a pair of posets for which the equality holds in Theorem \ref{main theorem}.
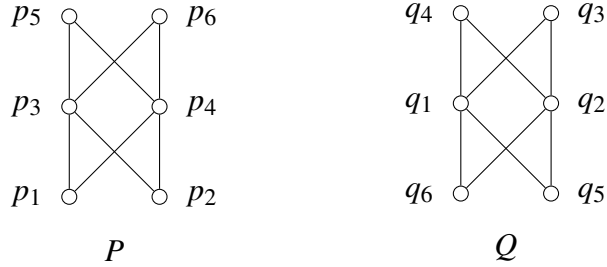
\begin{figure}\label{d=6}
\begin{center}
    \begin{tikzpicture}[scale=1.2]
        \node[draw, circle, fill=white, inner sep=2pt] (p1) at (0,0) {};
        \node[left=1mm of p1] {$p_1$};

        \node[draw, circle, fill=white, inner sep=2pt] (p2) at (1,0) {};
        \node[right=1mm of p2] {$p_2$};

        \node[draw, circle, fill=white, inner sep=2pt] (p3) at (0,1) {};
        \node[left=1mm of p3] {$p_3$};

        \node[draw, circle, fill=white, inner sep=2pt] (p4) at (1,1) {};
        \node[right=1mm of p4] {$p_4$};

        \node[draw, circle, fill=white, inner sep=2pt] (p5) at (0,2) {};
        \node[left=1mm of p5] {$p_5$};

        \node[draw, circle, fill=white, inner sep=2pt] (p6) at (1,2) {};
        \node[right=1mm of p6] {$p_6$};

        \draw (p1) -- (p3);
        \draw (p1) -- (p4);
        \draw (p2) -- (p3);
        \draw (p2) -- (p4);
        \draw (p3) -- (p5);
        \draw (p3) -- (p6);
        \draw (p4) -- (p5);
        \draw (p4) -- (p6);

        \node at (0.5,-0.6) {$P$};
    \end{tikzpicture}
    \hspace{2cm}
    \begin{tikzpicture}[scale=1.2]
        \node[draw, circle, fill=white, inner sep=2pt] (q1) at (0,0) {};
        \node[left=1mm of q1] {$q_6$};

        \node[draw, circle, fill=white, inner sep=2pt] (q2) at (1,0) {};
        \node[right=1mm of q2] {$q_5$};

        \node[draw, circle, fill=white, inner sep=2pt] (q3) at (0,1) {};
        \node[left=1mm of q3] {$q_1$};

        \node[draw, circle, fill=white, inner sep=2pt] (q4) at (1,1) {};
        \node[right=1mm of q4] {$q_2$};

        \node[draw, circle, fill=white, inner sep=2pt] (q5) at (0,2) {};
        \node[left=1mm of q5] {$q_4$};

        \node[draw, circle, fill=white, inner sep=2pt] (q6) at (1,2) {};
        \node[right=1mm of q6] {$q_3$};

        \draw (q1) -- (q3);
        \draw (q1) -- (q4);
        \draw (q2) -- (q3);
        \draw (q2) -- (q4);
        \draw (q3) -- (q5);
        \draw (q3) -- (q6);
        \draw (q4) -- (q5);
        \draw (q4) -- (q6);

        \node at (0.5,-0.6) {$Q$};
    \end{tikzpicture}
\end{center}
\caption{Posets $P$ and $Q$ for which the equality holds in Theorem \ref{main theorem}}
\end{figure}
Note that the condition in Theorem \ref{main theorem} is the same as that
in the conjecture \cite[Conjecture 2]{BrBr} on the facets of symmetric edge polytopes.

It is worth mentioning the number of twinned chain polytopes $\Gamma (P, Q)$.
In Table \ref{table S/T}, we show the number of pairs of comparability graphs on $d$ vertices \cite[A123416]{OEIS}.
The chain polytope $\Cc(P)$ is determined by the comparability graph $G_P$.
However, permuting the labels of the vertices of $G_Q$ does not preserve the unimodular equivalence of $\Gamma (P, Q)$ (see Example \ref{rei} for details).
Hence, the number of twinned chain polytopes can be larger than the number of pairs of comparability graphs.
Although twinned chain polytopes constitute a small fraction of all reflexive polytopes, they form a broad class in their own right.

\begin{table}[ht]
    \centering
    \renewcommand{\arraystretch}{1.2}
    \begin{tabular}{|c|c|c|c|c|c|c|c|c|c|}
        \hline
        number of vertices & 2 & 3 & 4 & 5 & 6 & 7 & 8  \\ \hline
        comparability graphs & 2 & 4 & 11 & 33 & 144 & 824 & 6,793  \\ \hline
        pairs of comparability graphs & 3 & 10 & 66 & 561 & 10,440 & 339,900 & 23,075,821  \\ \hline
    \end{tabular}
    \smallskip
    \caption{Number of comparability graphs / pairs of comparability graphs}
    \label{table S/T}
\end{table}

The present paper is organized as follows.
In Section \ref{sec:review of facets}, we review the equation for the number of facets of $\Gamma(P, Q)$ and provide remarks on labeling.
We show that the ordinal sum of posets corresponds to the free sum of the associated twinned chain polytopes, and use this to determine $N(\Gamma(P, Q))$ when $P$ and $Q$ are elementary posets.
In Section \ref{sec:degree of vertices},
we prove several lemmas on 
degrees of vertices of the comparability graphs $G_P$ and $G_Q$.
Finally, in Section \ref{sec: proof}, after studying the $3$-dimensional case in detail, we prove Theorem \ref{main theorem}.

\vspace{5mm}

\section{The number of facets of twinned chain polytopes} \label{sec:review of facets}
In the present section, we review the equation of the number of facets of twinned chain polytopes and examine the number of facets based on their geometric structure when these polytopes are associated with elementary posets. 

Let $P =\{p_1,\dots,p_d\}$ and $Q = \{q_1,\ldots,q_d\}$ be posets with $P \cap Q = \emptyset$.
For $W \subset [d]:=\{1,\dots,d\}$, let $(\Delta_W(P,Q), \leq_W)$ be the ordinal sum of $P_W$ and $Q_{\overline{W}}$, where $\overline{W} = [d] \setminus W$.
Note that $|\Delta_W(P,Q)|=d$.
A subset $C \subset P$ is called a {\it chain} if $C$ is a chain as a subposet of $P$.
Note that the empty set is a chain of $P$.
A chain $C$ of $P$ is called {\it maximal} if it is not contained in a larger chain of $P$.
We write $\Mc(P)$ for the set of all maximal chains of $P$.
The following proposition will be frequently used throughout this paper.

\begin{prop}[{\cite[Theorem 3.2]{T}}]\label{facets}
Let $P =\{p_1,\dots,p_d\}$ and $Q = \{q_1,\ldots,q_d\}$ be posets with $P \cap Q = \emptyset$. 
Then
$$
N(\Gamma(P, Q))
=
|\bigcup_{W \subset [d]} \Mc(\Delta_W (P, Q)) |.$$
\end{prop}

\begin{remark}\label{label}
Maximal chains of a poset $P$ correspond to the maximal cliques of 
the comparability graph $G_P$ of $P$.
From Proposition \ref{facets}, $N(\Gamma(P, Q))=N(\Gamma(P', Q))$ if 
the map $\varphi: P=\{p_1,\ldots,p_d\} \rightarrow P'=\{p_1',\ldots,p_d'\}$ defined by $\varphi(p_i) = p_i'$ induces a graph isomorphism from $G_P$ to $G_{P'}$.
\end{remark}

The following example shows that a permutation of the labels of the vertices of $G_Q$ does not necessarily preserve $N(\Gamma(P, Q))$.

\begin{example}\label{rei}
     Let $(P, \le_P)$, $(Q, \le_Q)$ be posets on $P=\{p_1,p_2,p_3\}$ and $Q=\{q_1,q_2,q_3\}$ that are isomorphic to ${\bf 1} \oplus {\bf I}_2$.
    Then $G_P$ and $G_Q$ are paths with three vertices.
    Let $G_P$ be the path $(p_1,p_3,p_2)$.
    Then 
$G_Q$ is one of 
$(q_1,q_3,q_2)$,
$(q_1,q_2,q_3)$,
or
$(q_2,q_1,q_3)$.
If $G_Q = (q_2,q_1,q_3)$, 
by exchanging $p_1$, $q_1$
with $p_2$, $q_2$, 
this case reduces to the case
$G_Q = (q_1,q_2,q_3)$.
Hence, up to symmetries,
there are two possibilities for $G_Q$
(Figure~\ref{PQ} shows the pairs of posets $P$ and $Q$ for these cases).
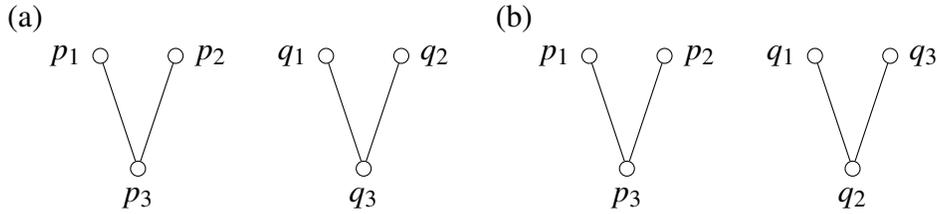
\begin{figure}[ht]
\centering
\begin{tikzpicture}[scale=1]

  \node at (-1, 2.5) {(a)};

  \node[circle, draw, inner sep=2pt, label=left:$p_1$] (p1a) at (0, 2) {};
  \node[circle, draw, inner sep=2pt, label=right:$p_2$] (p2a) at (1, 2) {};
  \node[circle, draw, inner sep=2pt, label=below:$p_3$] (p3a) at (0.5, 0.5) {};
  \draw (p1a) -- (p3a) -- (p2a);

  \node[circle, draw, inner sep=2pt, label=left:$q_1$] (q1a) at (3, 2) {};
  \node[circle, draw, inner sep=2pt, label=right:$q_2$] (q2a) at (4, 2) {};
  \node[circle, draw, inner sep=2pt, label=below:$q_3$] (q3a) at (3.5, 0.5) {};
  \draw (q1a) -- (q3a) -- (q2a);

  \node at (5.5, 2.5) {(b)};

  \node[circle, draw, inner sep=2pt, label=left:$p_1$] (p1b) at (6.5, 2) {};
  \node[circle, draw, inner sep=2pt, label=right:$p_2$] (p2b) at (7.5, 2) {};
  \node[circle, draw, inner sep=2pt, label=below:$p_3$] (p3b) at (7, 0.5) {};
  \draw (p1b) -- (p3b) -- (p2b);

  \node[circle, draw, inner sep=2pt, label=left:$q_1$] (q1b) at (9.5, 2) {};
  \node[circle, draw, inner sep=2pt, label=right:$q_3$] (q3b) at (10.5, 2) {};
  \node[circle, draw, inner sep=2pt, label=below:$q_2$] (q2b) at (10, 0.5) {};
  \draw (q1b) -- (q2b) -- (q3b);
\end{tikzpicture}
\caption{Pairs of posets $P, Q \cong {\bf 1} \oplus {\bf I}_2$ in Example \ref{rei}}
\label{PQ}
\end{figure}

    \begin{itemize}
        \item[(a)]
        If $G_Q=(q_1,q_3,q_2)$,
        then Table~\ref{table:a} is a list of $\Mc(\Delta_W (P, Q))$.
        Each set appears exactly once in the list.
        Hence we have $N(\Gamma(P, Q))=12$.

        \begin{table}[ht]
    \centering
    \scalebox{0.8}{
    \begin{tabular}{|c|c|c|c|c|c|c|c|c|}
    \hline
        $ W$ & $\emptyset$ & $\{1\}$ & $\{2\}$ & $\{3\}$ &  $\{1,2\}$ & $\{1,3\}$ & $\{2,3\}$ & \{1,2,3\}\\
        \hline
        $\Mc(\Delta_W (P, Q)) $ & $\begin{array}{c}
            \{q_1,q_3\} \\
             \{q_2,q_3\}
        \end{array}$ &  
        $\{p_1,q_2,q_3\}$&  
        $\{p_2,q_1,q_3\}$&
        $\begin{array}{c}
            \{p_3,q_1\} \\
             \{p_3,q_2\}
        \end{array}$&
        $\begin{array}{c}
            \{p_1,q_3\} \\
             \{p_2,q_3\}
        \end{array}$&
        $\{p_1,p_3,q_2\}$&
        $\{p_2,p_3,q_1\}$&
        $\begin{array}{c}
            \{p_1,p_3\} \\
             \{p_2,p_3\}
        \end{array}$\\
         \hline
    \end{tabular}
    }
    \smallskip
    \caption{The list of $\Mc(\Delta_W (P, Q))$ in Example \ref{rei} (a)}
    \label{table:a}
\end{table}

        \item[(b)]
        If $G_Q=(q_1,q_2,q_3)$,
        then
        Table~\ref{table:b} is a list of $\Mc(\Delta_W (P, Q))$.
        Note that $\{p_2,q_3\}$ appears twice in the list.
        Hence we have $N(\Gamma(P, Q))=11$.

         \begin{table}[ht]
    \centering
    \scalebox{0.8}{
    \begin{tabular}{|c|c|c|c|c|c|c|c|c|}
    \hline
        $ W$ & $\emptyset$ & $\{1\}$ & $\{2\}$ & $\{3\}$ &  $\{1,2\}$ & $\{1,3\}$ & $\{2,3\}$ & \{1,2,3\}\\
        \hline
        $\Mc(\Delta_W (P, Q)) $ & $\begin{array}{c}
            \{q_1,q_2\} \\
            \{q_2,q_3\}
        \end{array}$ &
         $\{p_1,q_2,q_3\}$&
        $\begin{array}{c}
            \{p_2,q_1\} 
            \\\{p_2,q_3\}
        \end{array}$&
         $\{p_3,q_1,q_2\}$&
        $\begin{array}{c}
            \{p_1,q_3\} \\ 
            \{p_2,q_3\}
        \end{array}$&
        $\{p_1,p_3,q_2\}$&
        $\{p_2,p_3,q_1\}$& 
        $\begin{array}{c}
            \{p_1,p_3\} \\
             \{p_2,p_3\}
        \end{array}$\\
         \hline
    \end{tabular}
    }
        \smallskip
    \caption{The list of $\Mc(\Delta_W (P, Q))$ in Example \ref{rei} (b)}
    \label{table:b}
\end{table}
    \end{itemize}

\end{example}

Let $\Pc \subset \RR^d$ be a $d$-dimensional polytope with $m$ facets, and $\Pc' \subset \RR^{d'}$ be a $d'$-dimensional polytope with $m'$ facets.
Suppose that $\Pc$ and $\Pc'$ have the origin in their interiors.
Then the $(d+d')$-dimensional polytope
\[\Pc \oplus \Pc' = \textrm{conv}(\Pc \times \{0\}^{d'} \cup \{0\}^{d} \times \Pc') \subset \RR^{d+d'}\]
is called the \textit{free sum} of $\Pc$ and $\Pc'$.
It is known that $\Pc \oplus \Pc'$ has $m \cdot m'$ facets (see \cite{HRZ}).

\begin{prop}\label{direct sum}
Let $(P_1, \le_{P_1})$, $(P_2, \le_{P_2})$, $(Q_1, \le_{Q_1})$ and $(Q_2, \le_{Q_2})$ be posets with 
$P_1=\{p_1,\dots,p_s\}$, $P_2 =\{p_{s+1},\dots, p_d\}$, $Q_1 = \{q_1,\dots,q_s\}$, and $Q_2 =\{q_{s+1},\dots,q_d\}$.
Then 
$$
N(\Gamma ( P_1 \oplus P_2, Q_1 \oplus Q_2 ))
=
N(\Gamma(P_1, Q_1)) 
\cdot
N(\Gamma(P_2, Q_2)) 
.
$$
In particular, if $N(\Gamma(P_1, Q_1)) \le 6^{\frac{s}{2}}$ and $N(\Gamma(P_2, Q_2)) \le 6^{\frac{d-s}{2}}$, then 
$$
N(\Gamma(P_1 \oplus P_2, Q_1 \oplus Q_2)) \le 6^{\frac{d}{2}}
.$$
\end{prop}

\begin{proof}
From \cite[Lemma 7.2]{HKT}, we have 
\begin{equation} \label{ds}
\Cc(P_1 \oplus P_2) \;\cup\; (-\Cc(Q_1 \oplus Q_2))
=
\Cc(P_1) \oplus \Cc(P_2) \;\cup\; (-\Cc(Q_1)) \oplus (-\Cc(Q_2)).
\end{equation}
Taking the convex hull of both sides of (\ref{ds}), we obtain 
\begin{align*}
& \ \ \ \ \; \Gamma ( P_1 \oplus P_2, Q_1 \oplus Q_2 )\\
&=
\textrm{conv} \Bigl(\bigl(\Cc(P_1) \times \{0\}^{d-s} \;\cup\; \{0\}^s \times \Cc(P_2)\bigr)
\;\cup\; 
\bigl((-\Cc(Q_1)) \times \{0\}^{d-s} \;\cup\;
\{0\}^s \times (-\Cc(Q_2))\bigr)\Bigr)\\
&=
\textrm{conv} \Bigl(
\bigl(\Cc(P_1) \;\cup\; (-\Cc(Q_1))\bigr) \times \{0\}^{d-s} 
\;\cup\;
\{0\}^s \times \bigl(\Cc(P_2) \;\cup\; (-\Cc(Q_2))\bigr)
\Bigr)\\
&=
\textrm{conv} \Bigl(\Gamma(P_1, Q_1) \times \{0\}^{d-s} 
\;\cup\;
\{0\}^s \times \Gamma(P_2, Q_2) \Bigr)\\
&=
\Gamma(P_1, Q_1) \oplus \Gamma(P_2, Q_2).
\end{align*}

Thus we have
$$
N(\Gamma(P_1 \oplus P_2, Q_1 \oplus Q_2))
=
N(\Gamma(P_1, Q_1)) \cdot N(\Gamma(P_2, Q_2)).
$$
If $N(\Gamma(P_1, Q_1)) \le 6^{\frac{s}{2}}$ and $N(\Gamma(P_2, Q_2)) \le 6^{\frac{d-s}{2}}$, then
$
N(\Gamma(P_1 \oplus P_2, Q_1 \oplus Q_2)) 
\le 
 6^{\frac{s}{2}} \cdot  6^{\frac{d-s}{2}}
 =6^{\frac{d}{2}}.
$
\end{proof}

Theorem \ref{main theorem} holds for the following posets.

\begin{prop}
\label{CCIIIC}
Let $(P, \le_P)$ and $(Q, \le_Q)$ be posets with $|P| = |Q| = d$.

\begin{itemize}
\item[(1)]
If $P, Q\;\cong\;{\bf C}_d$, then 
$N(\Gamma (P, Q)) = 2^d < 14 \cdot 6^{\frac{d-3}{2}}$.

\item[(2)]
If $P, Q\;\cong\;{\bf I}_d$, then 
$N(\Gamma (P, Q)) = d^2 +d$.
If $d=2$, then $d^2+d=  6^\frac{2}{2}$.
If $d \ne 2$, then $d^2+d < 14 \cdot  6^{\frac{d-3}{2}}$.

\item[(3)]
If $P\;\cong\;{\bf I}_d$, $Q\;\cong\;{\bf C}_d$, then 
$N(\Gamma(P, Q)) = d \cdot 2^{d-1} + 1< 14 \cdot 6^{\frac{d-3}{2}}$.
\end{itemize}

\end{prop}

\begin{proof}
(1)
From Proposition \ref{direct sum}, we have
$$
N(\Gamma ({\bf C}_d , {\bf C}_d ))
=
N(\Gamma ({\bf 1} \oplus \cdots \oplus {\bf 1}, {\bf 1} \oplus \cdots \oplus {\bf 1}))
=
N(\Gamma ({\bf 1}, {\bf 1}))^d = 2^d.
$$
Moreover 
$2^d = 2 \cdot 2^{d-1} < \frac{7}{3} \cdot \sqrt{6}^{d-1} =  14 \cdot  6^{\frac{d-3}{2}}$.

(2)
We define the partition 
$$
\bigcup_{W \subset [d]} \Mc(\Delta_W(P, Q)) = \Mc_1 \sqcup \Mc_2 \sqcup \Mc_3,
$$
where
$$
\Mc_1 = \Mc(\Delta_{\emptyset}(P, Q)), \quad
\Mc_2= \Mc(\Delta_{[d]}(P, Q))\quad \mbox{and} \quad
\Mc_3 = \bigcup_{\emptyset \neq W \subsetneq [d]} \Mc(\Delta_W(P, Q)).
$$
Since $|\Mc_1|=|\Mc_2|=d$ and $|\Mc_3|=d(d-1)$, the assertion follows from Proposition \ref{facets}.
Let $\varphi(d)=d^2+d$ and $\psi(d) = 14 \cdot 6^{\frac{d-3}{2}}$.
Then $\varphi(1) = 2 < 7/3 = \psi(1)$ and $\varphi(2) = 6$.
If $d \ge 3$, then
$$
\frac{\varphi(d+1)}{\varphi(d)} = 1+\frac{2}{d} < \sqrt{6} = \frac{\psi(d+1)}{\psi(d)}.
$$
Since $\varphi(3) = 12 < 14 = \psi(3)$, we have $\varphi(d) < \psi(d)$  for $d \ge 3$.

(3)
It is known \cite[Example 3.5]{T} that $N(\Gamma(P, Q)) = d \cdot 2^{d-1} + 1$.
Let $\varphi(d)=d \cdot 2^{d-1} + 1$ and $\psi(d) = 14 \cdot 6^{\frac{d-3}{2}}$.
It is easy to check that $\varphi(d) < \psi(d)$ for $d \le 5$.
If $d \ge 5$, then 
$$
\frac{\varphi(d+1)}{\varphi(d)} = \frac{(d+1) \cdot 2^d+1}{d \cdot 2^{d-1}+1} 
=
2+\frac{2}{d} - \frac{d+2}{d \left(d \cdot 2^{d-1}+1\right)} <
2+\frac{2}{d}< \sqrt{6} = \frac{\psi(d+1)}{\psi(d)}.
$$
Thus $\varphi(d) < \psi(d)$ for $d \ge 6$.
\end{proof}

In particular, Theorem \ref{main theorem} holds for $d \le 2$.
(Figure~\ref{d2Tw} shows all $\Gamma (P, Q)$ for $d=2$.)

\begin{cor}\label{d2}
Let $(P, \le_P)$ and $(Q, \le_Q)$ be posets with $|P| = |Q| = 2$.
Then we have $N(\Gamma (P, Q)) \le 6$.
The equality holds if and only if $P$
and $Q$ are isomorphic to ${\bf I}_2$.
\end{cor}

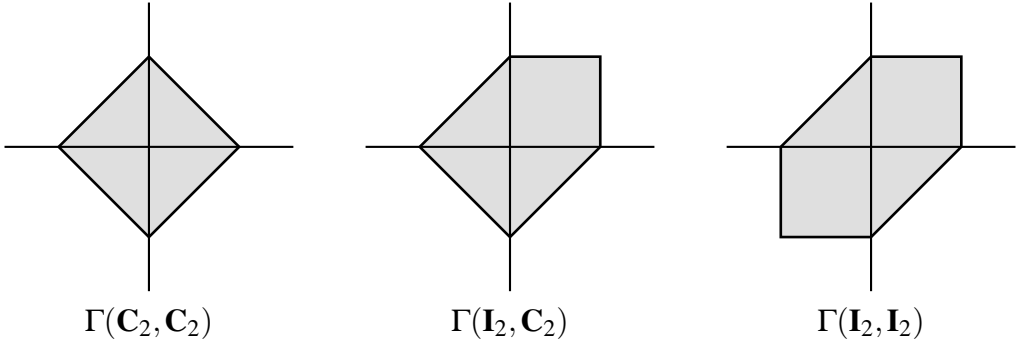
\begin{figure}[ht]
\centering
\begin{tikzpicture}[scale=1.2,
  axis/.style={
    line width=0.8pt,
  },
  poly/.style={line width=1.0pt, draw=black, fill=gray!25}
]
\begin{scope}[shift={(0,0)}]
  \draw[poly] (-1,0) -- (0,1) -- (1,0) -- (0,-1) -- cycle;
  \draw[axis] (-1.6,0) -- (1.6,0);
  \draw[axis] (0,-1.6) -- (0,1.6);
  \node at (0,-1.95) {$\Gamma({\bf C}_2,{\bf C}_2)$};
\end{scope}

\begin{scope}[shift={(4.0,0)}]
  \draw[poly] (-1,0) -- (0,1) -- (1,1) -- (1,0) -- (0,-1) -- cycle;
  \draw[axis] (-1.6,0) -- (1.6,0);
  \draw[axis] (0,-1.6) -- (0,1.6);
  \node at (0,-1.95) {$\Gamma({\bf I}_2,{\bf C}_2)$};
\end{scope}

\begin{scope}[shift={(8.0,0)}]
  \draw[poly] (-1,0) -- (0,1) -- (1,1) -- (1,0) -- (0,-1) -- (-1,-1) -- cycle;
  \draw[axis] (-1.6,0) -- (1.6,0);
  \draw[axis] (0,-1.6) -- (0,1.6);
  \node at (0,-1.95) {$\Gamma({\bf I}_2, {\bf I}_2)$};
\end{scope}
\end{tikzpicture}
\caption{Twinned chain polytopes for $d=2$}
\label{d2Tw}
\end{figure}

\begin{cor}\label{I_2}
Let $(P, \le_P)$ and $(Q, \le_Q)$ be posets with $|P|=|Q|=d$.
If $P,\; Q \cong {\bf I}_2 \oplus \cdots \oplus {\bf I}_2$, and 
$\{p_{2i-1}, p_{2i}\}, \{q_{2i-1}, q_{2i}\}
\cong {\bf I}_2 $ for $1 \le i \le d/2$,
then $N(\Gamma (P, Q)) = 6^{\frac{d}{2}}$.
\end{cor}

\begin{proof}
It follows from Proposition \ref{direct sum} and Proposition \ref{CCIIIC} (2).
\end{proof}

The following lemma follows from a well-known fact in graph theory.

\begin{lemma}[\cite{MoMo}]
\label{maxchain}
Let $(P, \le_P)$ be a poset with $|P| = d$.
Then
$
|\Mc(P)| \le \lfloor 3^{\frac{d}{3}} \rfloor. 
$
   
\end{lemma}

From Lemma \ref{maxchain}, one can show that
Theorem \ref{main theorem} holds when $P$ is a chain.

\begin{prop}\label{P=chain}
Let $P=\{p_1,\dots, p_d\}$ and 
$Q=\{q_1,\dots,q_d\}$ be posets with $P \cap Q = \emptyset$.
Suppose that either $\deg_{G_P} p_i =d-1$ or $\deg_{G_Q} q_i =d-1$ for each $i= 1,2,\dots,d$.
Then 
\begin{equation*}
    N(\Gamma (P, Q)) <
    \left\{
    \begin{array}{cl}
    6^{\frac{d}{2}} &  \mbox{if } d \mbox{ is even},\\
    14 \cdot 6^{\frac{d-3}{2}}& \mbox{if } d \mbox{ is odd}.\\
    \end{array}
    \right.
\end{equation*}
\end{prop}

\begin{proof}
We may assume that $d \ge 3$.
Suppose that
$\deg_{G_P} p_i =d-1$ for $i =1,2,\dots,m$
and  $\deg_{G_Q} q_j =d-1$ for $j =m+1,m+2,\dots,d$.
Let $W \subset [d]$.
We fix an arbitrary maximal chain $C \in \Mc(\Delta_W (P, Q) )$. 
Since $\deg_{G_P} p_i =d-1$ for any $i \in [m]$, we have $P_{W \cap [m]} \subset C$.
Similarly, since $\deg_{G_Q} q_j =d-1$ for $j =m+1,m+2,\dots,d$,
we have $Q_{\overline{W} \cap \{m+1,\dots,d\}} \subset C$.
On the other hand, 
$$C \setminus ( P_{W \cap [m]} \cup Q_{\overline{W} \cap  \{m+1,\dots,d\}}  )$$
is a maximal chain of 
$$P_{W \cap \{m+1,\dots,d\}} \oplus Q_{\overline{W} \cap [m]}.$$
Therefore, 
\begin{equation}\label{Nchain}
N(\Gamma (P, Q))
=
\sum_{W \subset [d]}
|\Mc(P_{W \cap \{m+1,\dots,d\}})|
|\Mc(Q_{\overline{W} \cap [m]})|.
\end{equation}
Let 
$|\overline{W} \cap [m]| = k$ and $|W \cap \{m+1,\dots,d\}| = \ell$. 
By Lemma~\ref{maxchain}, 
we obtain
$$
N(\Gamma (P, Q))  \le 
\left(\sum_{k=0}^m \binom{m}{k} \lfloor3^\frac{k}{3}\rfloor \right) 
\left(\sum_{\ell=0}^{d-m} \binom{d-m}{\ell} \lfloor3^\frac{\ell}{3}\rfloor \right) 
\le
(1+\sqrt[3]{3})^d 
< 
6^{\frac{d}{2}} 
$$
since $1 + \sqrt[3]{3} < \sqrt{6}$.
Moreover, since $(1+\sqrt[3]{3})^{17}
< 
\frac{7\sqrt{6}}{18} \cdot 6^{\frac{17}{2}}
$, we have  $(1+\sqrt[3]{3})^{d}
< 
\frac{7\sqrt{6}}{18} \cdot 6^{\frac{d}{2}} 
$ for any $d \ge 17$.
For odd $d$ with $3 \le d \le 15$,
one can check
$$
N(\Gamma (P, Q))  
\le
\left(\sum_{k=0}^m \binom{m}{k} \lfloor3^\frac{k}{3}\rfloor \right) 
\left(\sum_{\ell=0}^{d-m} \binom{d-m}{\ell} \lfloor3^\frac{\ell}{3} \rfloor\right) \le 
\sum_{i=0}^d \binom{d}{i} \lfloor3^\frac{i}{3}\rfloor 
<  \frac{7\sqrt{6}}{18} \cdot 6^{\frac{d}{2}}
$$
\begin{table}
    \centering
    \begin{tabular}{|c|c|c|c|c|c|c|c|}
    \hline
$d$ & 3 & 5 & 7 & 9 & 11 & 13 & 15\\
\hline
$\sum_{i=0}^d \binom{d}{i} \lfloor3^\frac{i}{3}\rfloor $ & 13 & 82 & 496 & 2971 & 17756 & 106522 & 640651 \\
\hline
$14 \cdot 6^\frac{d-3}{2} \left(=\frac{7\sqrt{6}}{18} \cdot 6^{\frac{d}{2}}\right) $& 14 & 84 & 504 & 3024 & 18144 & 108864 & 653184 \\
\hline
     \end{tabular}
     \smallskip
    \caption{Numbers appearing in Proof of Proposition \ref{P=chain}}
    \label{tab:placeholder}
\end{table}
through explicit computations (see Table \ref{tab:placeholder}).
\end{proof}

The following corollary follows immediately from equation (\ref{Nchain}).
\begin{cor}\label{Pchainform}
If $P\;\cong\;{\bf C}_d$, then
$$
N(\Gamma (P, Q))
=
\sum_{W \subset [d]} |\Mc(Q_W)|
<
    \left\{
    \begin{array}{cl}
    6^{\frac{d}{2}} &  \mbox{if } d \mbox{ is even},\\
    14 \cdot 6^{\frac{d-3}{2}}& \mbox{if } d \mbox{ is odd}.\\
    \end{array}
    \right.
$$
\end{cor}

\section{Degrees of vertices of the comparability graphs}  \label{sec:degree of vertices}

In the present section, we prove several lemmas on 
degrees of vertices of the comparability graphs $G_P$ and $G_Q$.
Lemma \ref{most important} will be frequently used throughout this section.
Let $c(P)$ denote the number of chains in $P$.

\begin{lemma} \label{most important}
Let $P=\{p_1,\dots, p_d\}$ and 
$Q=\{q_1,\dots,q_d\}$ be posets with $P \cap Q = \emptyset$
and let
$$
\Mc = \left\{ C \in \bigcup_{W \subset [d]} \Mc(\Delta_W (P, Q) ) \;:\; p_d \in C \right\}.
$$
If $p_d$ is incomparable with $p_{i_1},\dots,p_{i_s}$ in $P$ 
with $0 \le s \le d-1$,
then we have
\begin{equation} \label{keykeykey}
    |\Mc| \le c(Q_{\{i_1,\dots,i_s\}}) \cdot N(\Gamma (P_{[d-1] \setminus \{i_1,\dots,i_s\}}, Q_{[d-1] \setminus \{i_1,\dots,i_s\}})).
\end{equation}
\end{lemma}

\begin{proof}
Let 
\begin{eqnarray*}
\Mc' &=& \bigcup_{W \subset [d-1] \setminus \{i_1,\dots,i_s\}} \Mc(\Delta_W (P_{[d-1] \setminus \{i_1,\dots,i_s\}}, Q_{[d-1] \setminus \{i_1,\dots,i_s\}}) ),\\
\Mc'' &=& 
\left\{ (C_1,C') : 
C_1 \mbox{ is a chain of } Q_{\{i_1,\dots,i_s\}},\  
C' \in  \Mc'
\right\}.
\end{eqnarray*}
It is enough to show that there exists a surjective map $\varphi:\Mc'' \rightarrow \Mc$.
Let $(C_1,C') \in \Mc''$.
Then 
$$C' \in \Mc(\Delta_W (P_{[d-1] \setminus \{i_1,\dots,i_s\}}, Q_{[d-1] \setminus \{i_1,\dots,i_s\}}) )$$
for some ${W \subset [d-1] \setminus \{i_1,\dots,i_s\}}$.
Let
$$C =\{p_d\} \sqcup C_1 \sqcup C_2 \sqcup C_3,
$$
where 
\begin{eqnarray*}
C_2 &=& \{ p \in C' \cap P : p \mbox{ is comparable with } p_d\},\\
C_3 &=& \{q \in C' \cap Q : \{q\} \cup C_1 \mbox{ is a chain of } Q\}.
\end{eqnarray*}
Then $C$ belongs to $\Mc(\Delta_{W'} (P, Q))$ where 
$W' = \{d\} \sqcup W \sqcup (\{i_1,\dots,i_s\} \setminus \{i : q_i \in C_1\})$. 
Hence $C \in \Mc$.
We define $\varphi$ by $\varphi((C_1,C'))=C$.

We now show that $\varphi$ is surjective.
Suppose that $C \in \Mc$.
Then $C \in \Mc(\Delta_{W'} (P, Q))$ for some 
$d \in W' \subset [d]$.
Let
\begin{eqnarray*}
    C_1 &:=& C \cap Q_{\{i_1,\dots,i_s\}},\\
    C_2 &:=& (C \cap P) \setminus \{p_d\},\\
    C_3 &:=& C \cap Q_{[d-1] \setminus \{i_1,\dots,i_s\}}.
\end{eqnarray*}
Then $C_1$ is a chain of $Q_{\{i_1,\dots,i_s\}}$, and there exists 
$C'\in \Mc(\Delta_W (P_{[d-1] \setminus \{i_1,\dots,i_s\}}, Q_{[d-1] \setminus \{i_1,\dots,i_s\}}) )$  
 where $W = W' \cap ([d-1] \setminus \{i_1,\dots,i_s\})$
 such that $C_2 \cup C_3 \subset C'$.
It follows that $\varphi((C_1,C'))=C$ and 
hence $\varphi$ is surjective.
Thus we have (\ref{keykeykey}), as desired.
\end{proof}

\begin{lemma} \label{try}
Let $P=\{p_1,\dots, p_d\}$ and 
$Q=\{q_1,\dots,q_d\}$ be posets with $P \cap Q = \emptyset$.
Suppose that $p_d$ is incomparable with $p_{i_1},\dots,p_{i_s}$ in $P$ and
$q_d$ is incomparable with 
$q_{j_1},\dots,q_{j_t}$ in $Q$.
Then we have
\begin{eqnarray} \label{revo}
N(\Gamma (P, Q)) &\le&
N(\Gamma (P_{[d-1]}, Q_{[d-1]}))\\ \notag
&&+
c(Q_{\{i_1,\dots,i_s\}}) \cdot N(\Gamma (P_{[d-1] \setminus \{i_1,\dots,i_s\}}, Q_{[d-1] \setminus \{i_1,\dots,i_s\}}))\\ \notag
&&+
c(P_{\{j_1,\dots,j_t\}}) \cdot N(\Gamma (P_{[d-1] \setminus \{j_1,\dots,j_t\}}, Q_{[d-1] \setminus \{j_1,\dots,j_t\}})).
\end{eqnarray}
In particular, if {\rm (\ref{even odd eq})} holds for any posets with $\le d-1$ elements, 
and 
if $st \ne 0$ and $s+t \ge 4$,
then 
\begin{equation*} 
    N(\Gamma (P, Q)) 
    \left\{
    \begin{array}{ccl}
    < &6^{\frac{d}{2}} &  \mbox{if } d \mbox{ is even},\\
    \le &14 \cdot 6^{\frac{d-3}{2}}& \mbox{if } d \mbox{ is odd}.\\
    \end{array}
    \right.
\end{equation*}
\end{lemma}

\begin{proof}
We define a partition 
$
\bigcup_{W \subset [d]} \Mc(\Delta_W (P, Q))  = \Mc_1 \sqcup \Mc_2 \sqcup \Mc_3
$
where
\begin{eqnarray*}
\Mc_1 &=& \left\{ C \in \bigcup_{W \subset [d]} \Mc(\Delta_W (P, Q) ) \;:\;
p_d , q_d \notin C\right\},\\
\Mc_2 &=& \left\{ C \in \bigcup_{W \subset [d]} \Mc(\Delta_W (P, Q) ) \;:\; p_d \in C \right\},\\
\Mc_3  &=&   \left\{ C \in \bigcup_{W \subset [d]} \Mc(\Delta_W (P, Q)) \;:\; 
q_d \in C\right\}.
\end{eqnarray*}
It is easy to see that $\Mc_1 \subset \bigcup_{W \subset [d-1]} \Mc(\Delta_W (P_{[d-1]}, Q_{[d-1]}))$ and hence
$$
|\Mc_1| \le N(\Gamma(P_{[d-1]}, Q_{[d-1]})).
$$
From Lemma \ref{most important},
\begin{eqnarray*}
|\Mc_2| 
&\le &
c(Q_{\{i_1,\dots,i_s\}}) \cdot N(\Gamma (P_{[d-1] \setminus \{i_1,\dots,i_s\}}, Q_{[d-1] \setminus \{i_1,\dots,i_s\}})),\\
|\Mc_3|  &\le& c(P_{\{j_1,\dots,j_t\}}) \cdot N(\Gamma (P_{[d-1] \setminus \{j_1,\dots,j_t\}}, Q_{[d-1] \setminus \{j_1,\dots,j_t\}})).
\end{eqnarray*}
Hence we have (\ref{revo}).
Thus
$$    N(\Gamma (P, Q))
\le 
6^\frac{d-1}{2}
+ 
2^s \cdot 6^\frac{d-s-1}{2}
+
2^t \cdot 6^\frac{d-t-1}{2}\\
= 
\left(1+
\left( \frac{2}{\sqrt{6}} \right)^s
+
\left( \frac{2}{\sqrt{6}} \right)^t \ 
\right)
6^{\frac{d-1}{2}}.
$$
Suppose that $st \ne 0$.
Then
$$
1+
\left( \frac{2}{\sqrt{6}} \right)^s
+
\left( \frac{2}{\sqrt{6}} \right)^t
\left\{
\begin{array}{cccrc}
\le& 1+ \frac{2}{\sqrt{6}} +\frac{2}{3} \cdot \frac{2}{\sqrt{6}} 
&< &\sqrt{6} &\mbox{if } s+t \ge 4,\\
\\
\le& 1+ \frac{2}{\sqrt{6}} +\frac{4}{9}
&< &\frac{7\sqrt{6}}{18} \cdot \sqrt{6} &\mbox{if } s+t \ge 5,\\
\\
=&1+ \frac{2}{3}+ \frac{2}{3} & = &\frac{7\sqrt{6}}{18} \cdot \sqrt{6} &\mbox{if } s=t=2.
\end{array}
\right.
$$

If $d$ is odd and $(s,t)=(1,3)$, then 
$$
    N(\Gamma (P, Q))
\le
6^\frac{d-1}{2}
+ 
2 \cdot \frac{7\sqrt{6}}{18} \cdot 6^\frac{d-2}{2}
+
2^3 \cdot\frac{7\sqrt{6}}{18} \cdot 6^\frac{d-4}{2}
=
\left( 
\frac{3}{7}+\frac{1}{3}+\frac{2}{9}
\right) \frac{7\sqrt{6}}{18} \cdot
6^{\frac{d}{2}}=\frac{62}{63} \cdot
\frac{7\sqrt{6}}{18} \cdot
6^{\frac{d}{2}}.
$$
\end{proof}

\begin{lemma}\label{st11}
Let $P=\{p_1,\dots, p_d\}$ and 
$Q=\{q_1,\dots,q_d\}$ be posets with $P \cap Q = \emptyset$.
Suppose that 
\begin{itemize}
    \item 
    $p_d$ is incomparable with $p_i$ and comparable with $p_{i'}$ for all $i' \in [d-1] \setminus \{i\}$,
    \item 
    $q_d$ is incomparable with $q_j$ and comparable with $q_{j'}$ for all $j' \in [d-1] \setminus \{j\}$,
    \item 
    {\rm (\ref{even odd eq})} holds for any posets with $\le d-1$ elements,
\end{itemize}
then 
$N(\Gamma (P , Q ))$ satisfies {\rm (\ref{even odd eq})}.
The equality
$N(\Gamma (P , Q )) = 6^\frac{d}{2}$
holds if and only if 
$d$ is even, 
$i=j$,
$N(\Gamma(P_{[d-1] \setminus \{i\}}, Q_{[d-1] \setminus \{i\}})) = 6^{\frac{d-2}{2}} $, $P= \{p_i, p_d\} \oplus P_{[d-1] \setminus \{i\}}$
and $Q =  \{q_i, q_d\} \oplus Q_{[d-1] \setminus \{i\}}$.
\end{lemma}

\begin{proof}
     We define 
\begin{eqnarray*}
\Mc_1 &=& \left\{ C \in \bigcup_{W \subset [d]} \Mc(\Delta_W (P, Q) ) \;:\;
p_d \in C\right\},\\
\Mc_2 &=& \left\{ C \in \bigcup_{W \subset [d]} \Mc(\Delta_W (P, Q) ) \;:\;
q_d \in C\right\},\\
\Mc_3 &=& \left\{ C \in \bigcup_{W \subset [d]} \Mc(\Delta_W (P, Q) ) \;:\;
p_i \in C, q_d \notin C\right\}\\
&=&
\left\{ C \in \bigcup_{W \subset [d-1]} \Mc(\Delta_W (P_{[d-1]}, Q_{[d-1]}) ) \;:\;
p_i \in C\right\},\\
\Mc_4 &=& \left\{ C \in \bigcup_{W \subset [d]} \Mc(\Delta_W (P, Q) ) \;:\;
q_j \in C, p_d \notin C\right\}\\
&=& \left\{ C \in \bigcup_{W \subset [d-1]} \Mc(\Delta_W (P_{[d-1]}, Q_{[d-1]}) ) \;:\;
q_j \in C \right\}.
\end{eqnarray*}
For any $C \in \bigcup_{W \subset [d]} \Mc(\Delta_W (P, Q) ) $, if $p_d , q_d \notin C$, then either $p_i$ or $q_j$ belongs to $C$.
Hence
$$\bigcup_{W \subset [d]} \Mc(\Delta_W (P, Q))  = \Mc_1 \sqcup \Mc_2 \sqcup 
(\Mc_3 \cup \Mc_4).$$
From Lemma \ref{most important},
\begin{eqnarray}\label{mmmm}
    |\Mc_1| & \le & 2 \cdot N(\Gamma(P_{[d-1] \setminus \{i\}}, Q_{[d-1] \setminus \{i\}})),\\ \notag
    |\Mc_2| & \le & 2 \cdot N(\Gamma(P_{[d-1] \setminus \{j\}}, Q_{[d-1] \setminus \{j\}})),\\ \notag
    |\Mc_3| & \le & N(\Gamma(P_{[d-1] \setminus \{i\}}, Q_{[d-1] \setminus \{i\}})),\\ \notag
    |\Mc_4| & \le & N(\Gamma(P_{[d-1] \setminus \{j\}}, Q_{[d-1] \setminus \{j\}})).    \notag
\end{eqnarray}
Thus we have 
\begin{eqnarray} \label{11eq}
N(\Gamma(P, Q)) &\le& 3 \cdot  N(\Gamma(P_{[d-1] \setminus \{i\}}, Q_{[d-1] \setminus \{i\}}))+ 3 \cdot N(\Gamma(P_{[d-1] \setminus \{j\}}, Q_{[d-1] \setminus \{j\}})) \\ \notag
&\le & 
\left\{
\begin{array}{cccl}
6 \cdot 6^\frac{d-2}{2} &=&  6^\frac{d}{2} & \mbox{if } d \mbox{ is even},\\
6 \cdot \frac{7\sqrt{6}}{18} \cdot 6^\frac{d-2}{2} &=& \frac{7\sqrt{6}}{18} \cdot 6^\frac{d}{2} & \mbox{if } d \mbox{ is odd}.
\end{array}\right.
\end{eqnarray}

Suppose that $N(\Gamma(P, Q)) =  6^\frac{d}{2}$.
Then $d$ is even and
\begin{itemize}
    \item 
    The equalities hold for the inequalities in (\ref{mmmm}),
    \item 
    $\Mc_3 \cap \Mc_4 = \emptyset$, and hence $i=j$, and    
    \item 
    $N(\Gamma(P_{[d-1] \setminus \{i\}}, Q_{[d-1] \setminus \{i\}}))=6^\frac{d-2}{2}$.
\end{itemize}
If $p_i$ is incomparable with $p_k$ for some $k \in [d-1]$, 
then 
$$|\Mc_3| \le
2 \cdot N(\Gamma(P_{[d-1] \setminus \{i,k\}}, Q_{[d-1] \setminus \{i,k\}}))
\le 2 \cdot \frac{7\sqrt{6}}{18} \cdot 6^\frac{d-3}{2} = \frac{7}{9} \cdot 6^\frac{d-2}{2}<  6^\frac{d-2}{2}$$
by Lemma \ref{most important}.
This is a contradiction.
Hence $p_i$ is comparable with $p_k$ for any $k \in [d-1]$.
Similarly, $q_j$ $(=q_i)$ is comparable with $q_k$ for any $k \in [d-1]$.
Thus $P= \{p_i, p_d\} \oplus P_{[d-1] \setminus \{i\}}$
and $Q =  \{q_i, q_d\} \oplus Q_{[d-1] \setminus \{i\}}$.
The converse follows from Proposition \ref{direct sum}.
\end{proof}

\begin{lemma}\label{st12}
Let $P=\{p_1,\dots, p_d\}$ and 
$Q=\{q_1,\dots,q_d\}$ be posets with $P \cap Q = \emptyset$.
Suppose that 
\begin{itemize}
    \item 
    $p_d$ is incomparable with $p_i$ and comparable with $p_{i'}$ for all $i' \in [d-1] \setminus \{i\}$,
    \item 
    $q_d$ is incomparable with $q_{j_1}$, $q_{j_2}$ and comparable with $q_{j'}$ for all $j' \in [d-1] \setminus \{j_1,j_2\}$,
    \item 
    {\rm (\ref{even odd eq})} holds for any posets with $\le d-1$ elements,
\end{itemize}
then 
\begin{equation*} 
    N(\Gamma (P, Q)) <
    \left\{
    \begin{array}{cl}
    6^{\frac{d}{2}} &  \mbox{if } d \mbox{ is even},\\
    14 \cdot 6^{\frac{d-3}{2}}& \mbox{if } d \mbox{ is odd}.\\
    \end{array}
    \right.
\end{equation*}
\end{lemma}

\begin{proof}  
    We define a partition 
$
\bigcup_{W \subset [d]} \Mc(\Delta_W (P, Q))  = \Mc_1 \sqcup \Mc_2 \sqcup \Mc_3 \sqcup \Mc_4,
$
where
\begin{eqnarray*}
\Mc_1 &=& \left\{ C \in \bigcup_{W \subset [d]} \Mc(\Delta_W (P, Q) ) \;:\;
p_d , q_d\notin C\right\},\\
\Mc_2 &=& \left\{ C \in \bigcup_{W \subset [d]} \Mc(\Delta_W (P, Q) ) \;:\;
p_d \in C\right\},\\
\Mc_3 &=& \left\{ C \in \bigcup_{W \subset [d]} \Mc(\Delta_W (P, Q) ) \;:\;
p_i, q_d \in C\right\},\\
\Mc_4 &=& \left\{ C \in \bigcup_{W \subset [d]} \Mc(\Delta_W (P, Q) ) \;:\;
q_d \in C, p_i \notin C\right\}.
\end{eqnarray*}
By the assumption for $p_d$ and $q_d$, it follows that 
$$\Mc_1 = \left\{ C \in \bigcup_{W \subset [d-1]} \Mc(\Delta_W (P_{[d-1]}, Q_{[d-1]}) ) \;:\;
\{p_i , q_{j_1}, q_{j_2}\} \cap  C \ne \emptyset\right\}.$$
On the other hand, $|\Mc_4| = |\Mc|$, where 
$$
\Mc=
\left\{
C \in \bigcup_{W \subset [d-1]} \Mc(\Delta_W (P_{[d-1]}, Q_{[d-1]}) ) \;:\;
\{p_i , q_{j_1}, q_{j_2}\} \cap  C = \emptyset\
\right\}
$$
since $\Mc_4 = \{C \cup \{q_d\} : C \in \Mc\}$.
Hence we have $|\Mc_1| + |\Mc_4| = |\Mc_1| + |\Mc| =  N(\Gamma (P_{[d-1]} , Q_{[d-1]} ))$.
From Lemma \ref{most important}, 
\begin{eqnarray*}
|\Mc_2| & \le & 2 \cdot N(\Gamma (P_{[d-1] \setminus \{i\}} , Q_{[d-1] \setminus \{i\}} )).
\end{eqnarray*}

\noindent
{\bf Case 1.} ($i \notin \{j_1,j_2\}$.)
First, we show
\begin{eqnarray} \label{12case1}
|\Mc_3| & \le & 4 \cdot N(\Gamma (P_{[d-1] \setminus \{i,j_1,j_2\}} , Q_{[d-1] \setminus \{i,j_1,j_2\}} ))
\end{eqnarray}
from a similar argument as in Proof of Lemma \ref{most important}.
Let
\begin{eqnarray*}
\Mc' &=& \bigcup_{W \subset [d-1] \setminus \{i,j_1,j_2\}} \Mc(\Delta_W (P_{[d-1] \setminus \{i,j_1,j_2\}}, Q_{[d-1] \setminus \{i,j_1,j_2\}}) ),\\
\Mc'' &=& 
\left\{ (C_1,C') : C' \in  \Mc', C_1 \mbox{ is a chain of } P_{\{j_1,j_2\}}
\mbox{ such that }
\{p_i\} \cup C_1 \mbox{ is a chain} 
\right\}.
\end{eqnarray*}
It is enough to show that there exists a surjective map $\varphi:\Mc'' \rightarrow \Mc_3$.
Let $(C_1,C') \in \Mc''$.
Then 
$$C' \in \Mc(\Delta_W (P_{[d-1] \setminus \{i,j_1,j_2\}}, Q_{[d-1] \setminus \{i,j_1,j_2\}}) )$$
for some ${W \subset [d-1] \setminus \{i,j_1,j_2\}}$.
Let
$$C =\{p_i, q_d\} \sqcup C_1 \sqcup C_2 \sqcup C_3,
$$
where 
\begin{eqnarray*}
C_2 &=& \{ q \in C' \cap Q : q \mbox{ is comparable with } q_d\},\\
C_3 &=& \{p \in C' \cap P : \{p, p_i\} \cup C_1 \mbox{ is a chain of } P\}.
\end{eqnarray*}
Then $C$ belongs to $\Mc(\Delta_{W'} (P, Q))$ where 
$W' = \{i\} \sqcup W \sqcup \{k : p_k \in C_1\}$. 
Hence $C \in \Mc_3$.
We define $\varphi$ by $\varphi((C_1,C'))=C$.

We now show that $\varphi$ is surjective.
Suppose that $C \in \Mc_3$.
Then $C \in \Mc(\Delta_{W'} (P, Q))$ for some 
$W' \subset [d-1]$.
Let
\begin{eqnarray*}
    C_1 &:=& C \cap \{p_{j_1},p_{j_2}\},\\
    C_2 &:=& (C \cap Q) \setminus \{q_d\},\\
    C_3 &:=& C \cap P_{[d-1] \setminus \{i,j_1,j_2\}}.
\end{eqnarray*}
Then $C_1$ is a chain of $\{p_{j_1},p_{j_2}\}$
such that
$\{p_i\} \cup C_1$ is a chain, and there exists 
$$C'\in \Mc(\Delta_W (P_{[d-1] \setminus \{i,j_1,j_2\}}, Q_{[d-1] \setminus \{i,j_1,j_2\}}) ),$$  
 where $W = W'  \setminus \{i,j_1,j_2\}$
 such that $C_2 \cup C_3 \subset C'$.
It follows that $\varphi((C_1,C'))=C$ and 
hence $\varphi$ is surjective.
Thus we have (\ref{12case1}).

Therefore, 
\begin{eqnarray*}
N(\Gamma (P , Q )) &\le &
N(\Gamma (P_{[d-1]} , Q_{[d-1]} )) +
2 \cdot N(\Gamma (P_{[d-1] \setminus \{i\}} , Q_{[d-1] \setminus \{i\}} ))\\
&&
+4 \cdot N(\Gamma (P_{[d-1] \setminus \{i,j_1,j_2\}} , Q_{[d-1] \setminus \{i,j_1,j_2\}} )).
\end{eqnarray*}
If $d$ is even, then
$$
N(\Gamma (P , Q )) 
\le \frac{7\sqrt{6}}{18}\cdot 6^\frac{d-1}{2} + 2 \cdot   6^\frac{d-2}{2}+4\cdot 6^\frac{d-4}{2}
= 
\left(
\frac{7}{18} + 
\frac{1}{3} + 
\frac{1}{9} 
\right) \cdot 6^\frac{d}{2}
= \frac{5}{6}\cdot 6^\frac{d}{2}
<6^\frac{d}{2}.
$$
If $d$ is odd, then
\begin{eqnarray*}
N(\Gamma (P , Q )) 
&\le& 6^\frac{d-1}{2} + 2 \cdot \frac{7\sqrt{6}}{18}\cdot 6^\frac{d-2}{2}+4\cdot \frac{7\sqrt{6}}{18} \cdot 6^\frac{d-4}{2}\\
&=& 
\left(
\frac{3}{7} + 
\frac{1}{3} + 
\frac{1}{9} 
\right) \cdot \frac{7\sqrt{6}}{18} \cdot 6^\frac{d}{2}\\
&=& \frac{55}{63}\cdot \frac{7\sqrt{6}}{18} \cdot 6^\frac{d}{2}< \frac{7\sqrt{6}}{18}\cdot 6^\frac{d}{2}.
\end{eqnarray*}

\noindent
{\bf Case 2.} ($i = j_1$.)
From Lemma \ref{most important}, 
we have
\begin{eqnarray*}
| \{ C \in \bigcup_{W \subset [d]} \Mc(\Delta_W (P, Q) ) \;:\; q_d \in C\} |
&\le &  c(P_{j_1,j_2}) \cdot N(\Gamma (P_{[d-1] \setminus \{j_1,j_2\}} , Q_{[d-1] \setminus \{j_1,j_2\}} )).
\end{eqnarray*}
Since the number of chains in $\{p_{j_1},p_{j_2}\}$
that contains $p_{j_1}$ is at most 2,
we have 
\begin{eqnarray*}
|\Mc_3| &\le & 2 \cdot N(\Gamma (P_{[d-1] \setminus \{j_1,j_2\}} , Q_{[d-1] \setminus \{j_1,j_2\}} )).
\end{eqnarray*}
Therefore, 
\begin{eqnarray*}
N(\Gamma (P , Q )) &\le & N(\Gamma (P_{[d-1]} , Q_{[d-1]} )) +
2 \cdot N(\Gamma (P_{[d-1] \setminus \{i\}} , Q_{[d-1] \setminus \{i\}} ))\\
& & +\ 2 \cdot N(\Gamma (P_{[d-1] \setminus \{j_1,j_2\}} , Q_{[d-1] \setminus \{j_1,j_2\}} )).
\end{eqnarray*}
If $d$ is even, then
$$
N(\Gamma (P , Q )) \le  \frac{7\sqrt{6}}{18} \cdot 6^\frac{d-1}{2} + 2 \cdot 6^\frac{d-2}{2}+2\cdot   \frac{7\sqrt{6}}{18} \cdot 6^\frac{d-3}{2}
= 
\left(
\frac{7}{18} + 
\frac{1}{3} + 
\frac{7}{54} 
\right) \cdot  6^\frac{d}{2}
= \frac{23}{27}\cdot  6^\frac{d}{2}
<\  6^\frac{d}{2}.
$$
If $d$ is odd, then
$$
N(\Gamma (P , Q ))  \le  6^\frac{d-1}{2} + 2 \cdot  \frac{7\sqrt{6}}{18}\cdot 6^\frac{d-2}{2}+2\cdot  6^\frac{d-3}{2} 
 = 
\left(
\frac{3}{7} + 
\frac{1}{3} + 
\frac{1}{7} 
\right) \cdot  \frac{7\sqrt{6}}{18} \cdot 6^\frac{d}{2}
=   \frac{19}{21}\cdot  \frac{7\sqrt{6}}{18} \cdot 6^\frac{d}{2}
< \frac{7\sqrt{6}}{18} \cdot 6^\frac{d}{2},
$$
as desired.
\end{proof}

\section{Proof of Theorem \ref{main theorem}} \label{sec: proof}

In the present section, we prove Theorem~\ref{main theorem}.
Before that, we will show that, if $d=3$, then $N(\Gamma (P, Q))$
is less than 
$14 \cdot 6^{\frac{3-3}{2}} =14$.

\begin{prop} \label{d3}
Let $(P, \le_P)$ and $(Q, \le_Q)$ be posets with $|P|=|Q|=3$.
Then we have 
$$N(\Gamma (P, Q)) \le 13.$$
The equality holds if and only if one of $P$
and $Q$ is isomorphic to ${\bf I}_3$ and the other is isomorphic to
one of ${\bf C}_3$, ${\bf 1} \oplus {\bf I}_2$ and ${\bf I}_2 \oplus {\bf 1}$.
\end{prop}

\begin{proof}
The set of all posets with three elements is $\{{\bf 1} + {\bf C}_2, {\bf C}_3, {\bf I}_3, {\bf 1} \oplus {\bf I}_2,{\bf I}_2 \oplus {\bf 1}\}$.
Since ${\bf 1} \oplus {\bf I}_2$ and ${\bf I}_2 \oplus {\bf 1}$
have the same comparability graph,
we may assume that $P$ and $Q$ belong to $\{{\bf 1} + {\bf C}_2, {\bf C}_3, {\bf I}_3, {\bf 1} \oplus {\bf I}_2\}$.

\bigskip

\noindent
{\bf Case 1.} ($Q =Q_{\{1\}} + Q_{\{2,3\}} \cong  {\bf 1} + {\bf C}_2$.)
If $P$ is isomorphic to either ${\bf 1} + {\bf C}_2$ or ${\bf 1} \oplus {\bf I}_2$, at least one of $\deg_{G_P} p_2$ or $\deg_{G_P} p_3$ is one.
We may assume that $\deg_{G_P} p_3=1$.
Since $\deg_{G_Q} q_3=1$, 
from (\ref{11eq}) in Proof of Lemma \ref{st11}, we have
$
  N(\Gamma (P, Q))
  \le
 6 \cdot N(\Gamma ({\bf 1}, {\bf 1}))
 = 12.
$
Suppose that $P$ is isomorphic to ${\bf I}_3$.
Then $\deg_{G_P} p_1=\deg_{G_Q} q_1=0$.
From Lemma \ref{try}, we have 
$
  N(\Gamma (P, Q))
  \le 
  N(\Gamma (P_{\{2,3\}}, Q_{\{2,3\}}))
  + c(P_{\{2,3\}}) +c(Q_{\{2,3\}})
  =5+3+4=12.
$
If $P$ is isomorphic to ${\bf C}_3$, then 
$N(\Gamma (P, Q))=11$
by Corollary~\ref{Pchainform}. 

\medskip 

\noindent
{\bf Case 2.} ($P, Q \not\cong  {\bf 1} + {\bf C}_2$.)

From Proposition \ref{CCIIIC},
$$N(\Gamma ({\bf C}_3, {\bf C}_3))=8,\ 
N(\Gamma ({\bf I}_3, {\bf I}_3))=12, \ 
N(\Gamma ({\bf I}_3, {\bf C}_3))=13.
$$
Thus we may assume that $P \not\cong {\bf C}_3, {\bf I}_3$.
Then $P$ is isomorphic to ${\bf 1} \oplus {\bf I}_2$.
If $Q$ is isomorphic to ${\bf 1} \oplus {\bf I}_2$, then $N(\Gamma (P, Q)) = 11 \mbox{ or } 12$ from Example \ref{rei}.
If $Q \cong {\bf C}_3$ ($={\bf 1} \oplus {\bf C}_2$),
then 
$$N(\Gamma (P, Q)) = N(\Gamma ({\bf 1}, {\bf 1})) \cdot N(\Gamma ({\bf I}_2, {\bf C}_2))
=2 \cdot 5 =  10$$ from Proposition \ref{direct sum}.
Suppose that $Q \cong {\bf I}_3$.
Let $P = \{p_1,p_2,p_3\}$ where $p_1$ and $p_2$ are not comparable.
Then $\bigcup_{W \subset [3]} \Mc(\Delta_W (P, Q))$ consists of
$$
\{q_1\}, \{q_2\}, \{q_3\}, 
\{p_1,q_2\}, 
\{p_1,q_3\}, 
\{p_2,q_1\}, 
\{p_2,q_3\}, 
\{p_3,q_1\}, 
\{p_3,q_2\},
$$
$$
\{p_1,p_3\},
\{p_1,p_3,q_2\},
\{p_2,p_3\},
\{p_2,p_3,q_1\}.
$$
Hence $N(\Gamma(P, Q)) = 13$.
(Table~\ref{d=3} shows the values of $N(\Gamma(P, Q))$ in Proof of Proposition~\ref{d3}.)
\end{proof}

\begin{table}[ht]
\centering
\renewcommand{\arraystretch}{1.2}
\begin{tabularx}{0.6\linewidth}{|C|C|C|C|C|}
\hline
\diagbox[height=1.5\line]{$Q$\ }{$P$} & ${\bf 1} + {\bf C}_2$ & ${\bf C}_3$ & ${\bf I}_3$ & ${\bf 1} \oplus {\bf I}_2$ \\ 
\hline
${\bf 1} + {\bf C}_2$ &$\le 12$ & $11$ & $\le 12$ & $\le 12$ \\ 
\hline
${\bf C}_3$ &  & $8$ & $13$ & $10$ \\ 
\hline
${\bf I} _3$ &  &  & $12$ & $13$ \\ 
\hline
${\bf 1} \oplus {\bf I}_2$ &  &  &  & $11 \text{ or } 12$ \\ 
\hline
\end{tabularx}
\smallskip
\caption{The values of $N(\Gamma(P, Q))$ for posets with $|P|=|Q|=3$}
\label{d=3}
\end{table}

We note here that the twinned chain polytope and the symmetric edge polytope are distinctly different classes of lattice polytopes.

\begin{remark}
Let $G$ be a finite simple graph on the vertex set $[d+1]$ with the edge set $E(G)$. 
The \textit{symmetric edge polytope $\Pc_G$} of $G$ is the convex hull of
$
 \{ \pm (\eb_i - \eb_j) : \{i,j\} \in E(G) \},
$
where $\eb_i$ is the $i$-th unit coordinate vector in $\RR^{d+1}$. 
It is known that the symmetric edge polytope of a connected graph with $d+1$ vertices is a centrally symmetric
reflexive $d$-dimensional polytope.
\begin{itemize}
    \item 
The number of vertices of $\Gamma({\bf I}_3, {\bf I}_3)$ is 14 by \cite[Proposition 3.3]{CFS}.
On the other hand, if $d = 3$, then
the number of vertices of $\Pc_G$ is less than or equal to $2 \cdot \binom{4}{2}=12$
and the equality holds if and only
if $G$ is the complete graph $K_4$. 
Thus, $\Gamma({\bf I}_3, {\bf I}_3)$ is not isomorphic to $\Pc_G$ for any graph $G$.

\item 
From \cite[Proposition 3.5]{HJM}, we have $N(\Pc_{K_4}) = 14$.
On the other hand, if $d = 3$, then $N(\Gamma(P, Q)) \le 13$ by Proposition \ref{d3}.
Hence, $\Pc_{K_4}$ is not isomorphic to $\Gamma(P, Q)$
for any posets $P$ and $Q$.
    
\end{itemize}
\end{remark}

We now prove the main theorem of this paper.

\begin{proof}[Proof of Theorem \ref{main theorem}]
By Corollary \ref{d2} and Proposition \ref{d3},
the assertion holds if $d \le 3$.
The proof proceeds by induction.
Let $d \ge 4$.
From Proposition~\ref{P=chain}, we may assume that
there exists $k \in [d]$ such that
$p_k$ is incomparable with exactly $s$ elements in $P$
and $q_k$ is incomparable with exactly $t$ elements in $Q$ with $st \ne 0$.
Then one of the following holds:
\begin{itemize}
    \item[(i)]
    $s + t \ge 4$,
    \item[(ii)]
    $(s,t) = (1,1)$,
    \item[(iii)]
    $(s,t) \in \{ (1,2), (2,1)\}$.    
\end{itemize}
If either (i) or (iii) holds, then   
\begin{equation*} 
    N(\Gamma (P, Q))
    \left\{
    \begin{array}{ccl}
    < &6^{\frac{d}{2}} &  \mbox{if } d \mbox{ is even},\\
    \le &14 \cdot 6^{\frac{d-3}{2}}& \mbox{if } d \mbox{ is odd}.\\
    \end{array}
    \right.
\end{equation*}
from Lemmas \ref{try} and \ref{st12}, respectively.
Suppose that (ii) holds. From Lemma \ref{st11}, 
$N(\Gamma (P , Q ))$ satisfies {\rm (\ref{even odd eq})}.
Moreover, for even $d$, the equality
$N(\Gamma (P , Q )) = 6^\frac{d}{2}$
holds if and only if 
\begin{itemize}
\item 
$N(\Gamma(P_{[d] \setminus \{i,k\}}, Q_{[d] \setminus \{i,k\}})) = 6^{\frac{d-2}{2}} $,
    \item
    $P= \{p_i, p_k\} \oplus P_{[d] \setminus \{i,k\}}$,
    \item 
    $Q =  \{q_i, q_k\} \oplus Q_{[d] \setminus \{i,k\}}$,
\end{itemize}
where $p_k$ (resp. $q_k$) is incomparable with $p_i$ (resp. $q_i$).
From
the hypothesis of induction, we have
$P= \{p_i, p_k\} \oplus P_{[d] \setminus \{i,k\}} \cong \{p_i, p_k\} \oplus {\bf I}_2 \oplus \dots \oplus {\bf I}_2$, $Q =  \{q_i, q_k\} \oplus Q_{[d] \setminus \{i,k\}}\cong \{q_i, q_k\} \oplus {\bf I}_2 \oplus \dots \oplus {\bf I}_2$,
and the map $\varphi: P \rightarrow Q$ defined by $\varphi(p_i) = q_i$ induces a graph  isomorphism from $G_P$ to $G_Q$.
(See Remark~\ref{label} for the statement about the map from the poset $P$ to $Q$.)
\end{proof}


\subsection*{Acknowledgment}
This work was supported by JSPS KAKENHI 24K00534.


\end{document}